\documentclass[psamsfonts]{amsart}

\usepackage{amssymb,amsfonts}
\usepackage[all,arc]{xy}
\usepackage{enumerate}
\usepackage{mathrsfs}
\usepackage{mathtools}
\usepackage{amsmath}
\usepackage{textcomp}
\usepackage{graphicx}
\usepackage[utf8]{inputenc}
\usepackage[T1]{fontenc}

\newtheorem{thm}{Theorem}[section]
\newtheorem{cor}[thm]{Corollary}
\newtheorem{prop}[thm]{Proposition}
\newtheorem{lem}[thm]{Lemma}

\theoremstyle{definition}
\newtheorem{defn}[thm]{Definition}

\theoremstyle{remark}
\newtheorem{rem}[thm]{Remark}

\newtheoremstyle{case}{}{}{}{}{}{:}{ }{}
\theoremstyle{case}
\newtheorem{case}{Case}

  \newbox\gnBoxA
\newdimen\gnCornerHgt
\setbox\gnBoxA=\hbox{$\ulcorner$}
\global\gnCornerHgt=\ht\gnBoxA
\newdimen\gnArgHgt
\def\Godelnum #1{%
\setbox\gnBoxA=\hbox{$#1$}%
\gnArgHgt=\ht\gnBoxA%
\ifnum     \gnArgHgt<\gnCornerHgt \gnArgHgt=0pt%
\else \advance \gnArgHgt by -\gnCornerHgt%
\fi \raise\gnArgHgt\hbox{$\ulcorner$} \box\gnBoxA %
\raise\gnArgHgt\hbox{$\urcorner$}}

\makeatletter
\let\c@equation\c@thm
\makeatother
\numberwithin{equation}{section}

\title{On the Uniqueness Problem for Notations of Recursive Ordinals}

\author{Matthew Timothy Wright}

\begin{document}

\begin{abstract}

In \textit{Ordinal Logics and the Characterizations of the Informal Concept of Proof}, Georg Kreisel poses the problem of assigning unique notations to recursive ordinals, and additionally suggests that the methods which are developed for its solution will be non-constructive in character. In this paper we develop methods in which various uniqueness results for notations of recursive ordinals can be obtained, and thereafter apply these results to investigate the problems surrounding the hierarchical classification of the computable functions.

\end{abstract}

\maketitle

\section{Introduction}

In [15, pg. 292] Kreisel addresses the problem of assigning unique notations to recursive ordinals, and suggests that the method applied to assigning these notations will be non-constructive. However, the main difficulties which appear in many attempts to resolve this problem are deeply related to various non-uniqueness results for hierarchies of computable functions indexed by recursive ordinals. Moreover, it has been pointed out by Feferman [6] that one of the higher goals of obtaining a satisfactory hierarchy of computable functions is to elucidate how to canonically classify any arbitrarily defined decision procedure with respect to some fixed level in the hierarchy---a hierarchy which we intuitively believe to be linearly ordered and everywhere defined.\footnote{Cf. Kanamori [11, pg. 256]} Another hope in achieving a unique hierarchy, as described in a remark by Wainer [1, pg. 150] on classes of ``verifiably'' provable functions, is that if such a hierarchy were indexed by unique notations for all recursive well-orderings, then the arithmetical information described in the structure of the hierarchy would shed light on the question of what it means to possess a ``natural'' well-ordering of $\omega$, especially when the focus of its uniqueness is placed on the constructive information contained in classifying the provably recursive functions of a formal theory.

Perhaps more subtly, a deeper conceptual problem arises when attempting to reach a purely hierarchical understanding of how assigning notations to recursive ordinals can be used to constructively generate a class of computable functions which is closed under relative computability. Naturally, this conceptual problem is explicitly encountered when attempting to distinguish the relative complexity of any pair of distinct, arbitrarily defined computable functions with respect to the decision problem of whether a recursive relation defines a well-ordering. Thus, when properly taken in the context of the apparent absoluteness\footnote{G{\"o}del [8, pg. 151] claims the ``absoluteness'' of the notion of computable function by citing the fact that it is invariant under adjoining higher types to any formal theory containing arithmetic with respect to diagonalization.} of the concept of mechanical procedure, these hopes illustrate the fundamental gaps which are encountered in any attempt to clarify questions regarding the concept of mechanical procedure and its formalization in any hierarchical manner.

However, before supplying further details, a brief overview of the intuitions involved in the direction of obtaining unique notations will be given. It is a well-known difficulty, in view of the negative results of various authors\footnote{See [4], [15], [17]}, that many hierarchies of computable functions indexed by recursive ordinals ``collapse'' at the first limit ordinal $\omega$, so that no \textit{unique} arithmetical information concerning the structure of the hierarchy can be measured beyond $\omega$. Further, by defining a one-one map from notations into the recursive ordinals via the system $\mathcal{O}$ of ordinal notations developed by Kleene [13, 14], a canonical difficulty can be immediately singled out. In particular, the arithmetically definable relation $<_e$ which stands between indices of recursive ordinals cannot be reduced to the relation $<_\mathcal{O}$ which stands between the notations for these recursive ordinals. That is, if $|\alpha|,|\beta|\in\mathcal{O}$ and $
\alpha,\beta$ are recursive ordinals, then 
\begin{equation}
\alpha<_e\beta\implies|\alpha|<_\mathcal{O}|\beta|
\end{equation}
but the converse does not necessarily hold, owing to the fact that $<_\mathcal{O}$ is a partially ordered $\Pi^1_1$-relation. Therefore, it is in this sense these hierarchies ``collapse'' because any recursive limit $\geq\omega$ can receive up to $2^\omega$ notations in $\mathcal{O}$, and thus one cannot give a non-trivial classification of the computable functions beyond the $\omega$th level by appealing to the order-type of the length of their termination proofs. Essentially, this failure of classification stalls any attempt at providing a unique, constructive meaning to the closure of the class of computable functions under diagonalization.

From these facts, the problem of assigning unique notations to recursive ordinals can be redressed as the problem of constructing an order-preserving relation which holds between distinct notations and does not suffer from the definability issues as sketched above. By studying the properties of $\mathcal{O}$, it becomes clear that any proposed system of notations that is constructed to overcome these issues cannot resemble the $\Pi^1_1$-complete structure of $\mathcal{O}$ in any outward way. Consequently, an immediate obstacle for defining an order-preserving relation which holds for all distinct notations is that one must jointly succeed in constructing a system $\mathcal{O}^*$ of notations such that $\alpha<_e\beta\iff|\alpha|<_\mathcal{O^*}|\beta|$, where $\mathcal{O}^*$ is arithmetically definable and all properties which we intuitively believe to hold for a natural hierarchy of computable functions (such as linearity and being everywhere defined) can be formally characterized within the structure of $\mathcal{O}^*$.

\section{What is a Natural Well-Ordering?}

To begin to discuss the question of what it means to possess a “natural” or “canonical” well-ordering of the integers, it becomes expedient to survey the conceptual issues at hand which are directly encountered in the attempt to make these matters more tractable. 

For Turing [25], it is left as a matter of intuition to verify, on the basis of mechanical inferences, whether an arbitrary recursive relation defines a well-ordering. However, independently of the difficulties involved in carrying out an effective verification (i.e., independent of one’s intuition), it is important to see that the constructive problems of supplying an effective verification find their origin within the lacuna of distinguishing between extensional and intensional measures of ordinal complexity for computable functions. In particular, the standard definition of a recursive ordinal is given within a purely constructive context; that is, effectively specifying a mechanical procedure which decides, in a computable number of steps, whether the recursive relation in question is a well-ordering, since any effectively enumerable set is computable if and only if its characteristic function is. Curiously, however, almost all constructive issues of this type appear to not depend on the property that any given recursive ordinal has a canonical representation. More concisely, the intensional nature of verifying that a recursive well-ordering is well-founded and the arithmetic statement\footnote{That is, the $\Pi^0_2$-condition which expresses that the procedure is everywhere defined does not depend on the $\Pi^1_1$-condition that the computaton tree of the procedure is well-founded. Thus, this well-foundedness condition is not arithmetically definable even if one allows arbitrarily long recursions of length $\geq \omega$ to define the procedure in question, implying that the problem of determing whether the computation tree is well-founded is not equivalent to having a witness to the statement that the procedure is everywhere defined at or before $\omega$. Cf. [22].} of the terminating procedures used to specify them seem to have no external relationship to one another.

As a result, to resolve the conceptual disagreement between these extensional and intensional concerns, one may attempt to separate out those effectively specified procedures which decide the totality of a recursive relation from those which enumerate the procedures according to their ordinal complexity. Consequently, we place our concern not on the extensional nature of these procedures, but rather on their purely intensional aspects. As a starting point, if we wish to clarify the conceptual issues involved, we intend to achieve some canonical description of the ordinal complexity of the effectively specifiable procedures involved in our enumeration, with the hope that such a description will mirror the complexity of verifying the totality of an arbitrarily given recursive relation. Unfortunately, an immediate stumbling-block in this direction is that the usual method of measuring the complexity of a computable function by means of constructively defined ordinals does not extend far enough to encompass a canonical classification of the class of everywhere defined decision procedures. As a consequence, what one requires of a canonical description of the ordinal complexity of these procedures is that the description should reflect the complexity of verification with the ``largeness'' of the order-type of the well-ordering that would be defined.\\

Thus the question to be resolved can be stated as follows: If one were able to obtain a canonical description of the complexity of any effectively specifiable procedure, then is there a method of associating this description with the recursive ordinals in a way that naturally reflects the order-type of the well-orderings that are to be defined? In essence, since the constructive concern of verifying the totality of a recursive relation appears to be independent of the requirement of having a ``natural'' representation of the ordinal, can one attempt to clarify the issue of obtaining canonical notations in a way that is \textit{intensionally} related to the ordinal complexity of verifying whether the decision procedure in question is everywhere defined?

On closer inspection, it seems that the lack of agreement between our constructive concerns and the intensional ambiguity of our analysis of the ordinal complexity of an arbitrarily defined decision procedure leads to a positive direction in which these questions can be resolved. In particular, because one is capable of excluding the constructive need for an effective verification from the want of an optimal measure for the ordinal complexity of a computable function, one can appeal to a certain non-constructive intuition that is implicit in Turing's analysis of effective calculability\footnote{See [26, pg. 249, Section 9.] for the intrinsic appeal provided by Turing`s Type (a) argument.}. That is, this non-constructive intuition is simply the belief that, independently of constructively verifying the totality of a recursive relation, there is the plausibility in the existence of a hierarchy of effectively specifiable mental procedures that are actualized in our experience of carrying out computations of varying degrees of difficulty. More directly, we have the evidence that, prior to judging the constructiveness of verifying that a (possibly non-terminating) procedure is total for deciding whether an arbitrary recursive relation defines a well-ordering, we are able to intuit that any method of ``measuring'' the complexity of this effective verification would be highly non-constructive, especially if one chooses to depend on the intuition that this hierarchy of decison procedures is everywhere defined. Consequently, we see that the origin of this hierarchical intuition can be interpreted in a purely non-constructive manner, and in this light, we are free to resolve any descriptive problem of analyzing the ordinal complexity of an arbitrary computable function without any dependence on the possible constructive nature of our evidence that it is everywhere defined. \\

However, one may demand that if such an intuition is to suit our difficulties, then one must characterize it in a manner that not only leads to a canonical measure of the ordinal complexity for arbitrary sequences of effectively specifiable procedures, but also require that it is capable of clarifying the logical definition of these procedures with respect to a certain formal theory. 
Thus one may recognize that, in order to develop this intuition for its possible use in a formal system, it is worthwhile to first cultivate its meaning in a way that is independent of a specified system of axioms. In this light, one may come to understand the development of this intuition not as a way of mechanically producing more evident theorems, but as a way of supplementing the concept of mechanical procedure and enriching the comprehensiveness of our non-constructive methods. On these few points it would seem that, to all appearances, the conscious application of such an intuition would inevitably consitute an appeal to evidence of a different and more decisive kind.

\section{Applications of Diagonally Non-Computable Majorizing Functions}

We shall rely heavily on the insightful texts of Jockusch and Soare [9] on diagonally non-recursive (DNR) functions and Sacks [23] on the theory of constructive ordinals and their applications.

Let $\sigma \in 2^{<\omega}$ denote a primitive recursive sequence number (under a G{\"o}del numbering) and refer to $2^{<\omega}$ as the set of all finite binary strings. Additionally, let $TOT \coloneqq \left\{e : \phi_{e}\ \text{is total} \right\}$. According to [10], we say that a function $F \in DNR_2$ is a $\left\{0,1\right\}$-valued diagonally non-computable function if and only if for all $e \in Dom(\phi)$, we have $F(e) \neq \phi_e(e)\downarrow$ such that $\phi_e(e)$ is the $e$th partial computable function on its $e$th input. Furthermore, it will common practice throughout to refer to $\phi_e(e)$ as the diagonal function. Additionally, we will have the following notation $f \simeq g$ to stand between two computable functions $f, g$  which mutually depend on the others definability. Finally, it will be important to note that, when provided with a suitable G{\"o}del numbering, we are able to represent $F \in DNR_2$ as an infinite $\Pi^0_1$-class. \\

Continuing in this direction, we aim to motivate this section by exploring some conceptual ideas related to diagonal functions and the class of arbitrary partial computable functions. As described by Feferman [6], there is a fundamental recursion-theoretic problem of finding a canonical classification for the non-constructively defined class of partial computable functions, and thereafter the concept of a majorizing function is introduced to supplement this idea, for which we have the following definition:
\begin{defn} (Feferman [4]) 
\textit{We write $g \ll f$ and say ``$f\ \text{majorizes}\ g$'' if and only if} $\qquad \qquad$ $$\exists y \forall x [y < x \implies g(x) < f(x)].$$
\end{defn}

A direction in which a solution of this classification problem can be attained begins by analyzing how one might determine ``from above'' whether every computable function $\psi$ which maps $\omega$ into $\omega$ and defines an $\textit{arbitrary}$ well-ordering of $\omega$ is everywhere defined. Essentially, we intend to determine if $\psi$ is everywhere defined by defining an increasing enumeration of a computable sequence $\left\{\phi_{e_i}\right\}_{i \in \omega}$ which ``globally'' reflects if $\psi$ is total computable, and aim to use the global information obtained by executing this enumeration to decide if any member contained in the sequence defines an infinite descending sequence on $\omega$ for all minimal elements\footnote{Equivalently, if every member belonging to this sequence is non-terminating.}.

Informally, we carry out this global enumeration by constructing a non-computable function $F$ which majorizes all members defined in $\left\{\phi_{e_i}\right\}_{i \in \omega}$ under some arbitrary indexing of the total computable functions. That is, we construct $F$ in a manner such that that all computable functions defined on the entire domain of a binary computable relation can be uniformly determined ``from above'' to be everywhere defined with respect to a simultaneous enumeration of a particular class of functions defined on all well-founded inital segments. By using this enumeration to diagonalize out of this class of functions, then one defines a certain ``almost everywhere'' diagonalization against all total functions if $F$ disagrees with the partial map $e \longmapsto \phi_e(e)\downarrow$ up to a constant for all $e \in \omega$, granted that the numbers in the domain of $F$ act as instances of fixed-points for the program which computes the well-ordering in a well-defined sense. In what follows, we aim to characterize the idea of determining that an arbitrary computable sequence is everywhere defined  by developing the notion of a ``diagonally non-computable" majorizing function for a class of everywhere defined decision procedures. \\

We begin with a brief overview of the elementary properties of $\left\{0,1\right\}$-valued computable functions
and the definition of a computation tree. Define the Kleene $T$-predicate as $T_n(e,x,t)$ such that, for
every $n \in \omega$, we have that the relation $T_n$ is used to define the triple $(e,x,t)$, with $e$ the index of a computable function $\phi_e(x)$, the numbers $ (x_1 ,\ldots, x_{n}) = x$ is a sequence of inputs, and the number $t$ is the G{\"o}del number which codes the finite sequence $\big \langle \sigma_0,\sigma_1,\ldots,\sigma_{n} \big \rangle$ of configurations of the computation yielding $\phi_e(x) = y$ by recursion on $e$. Again, let $t$ denote the computation tree of some computable function $\phi_e(x)$. Thus, if $\phi_e(x)$ is total, then it is total for some $e \in \omega$ by perfoming a recursion on the index such that $\phi_e(x)=y$ if we let $\phi_e(x)\downarrow$ mean that $\exists y, R_e(x,y)$ holds as total recursive relation. By definition, $\phi_e(x) \in \text{TOT} \iff \forall x \exists y$ such that $\phi_e(x) = y$ $\iff \forall x \exists t \exists y$ such that $T(e, x, t)$ holds and $U(e, x, t)$ computes $\varphi(e,x)$ for all $x \in \omega$ for which $\phi_e(x)\downarrow$.

\begin{thm}
(Normann [18]) \textit{For every $n \in \omega$, we have}:
\begin{enumerate}
\item \textit{$T_n$ is primitive recursive}.
\item \textit{There exists a primitive recursive function $U$ such that, if $t \in \omega$ is the computation tree of $\phi_{e}(x)$, then $U(t)$ returns the output of the corresponding computation (i.e., the terminal node of the halting computation yielding $U(e, x, t) = y)$}.
\item $\phi_e(x) = U(\min_{|t|} T_n(e,x,t))$.
\end{enumerate}
\end{thm}

Note that the proof of clause (1) requires that, if $t$ is the G{\"o}del number which codes the finite sequence of total configuration states in the computation of $\phi_e(x)$, then the monotonicity of the G{\"o}del numbering of $t$ is given by a partial computable function which is uniformly computable in the index which enumerates $\phi_e(x)$ for all instances. This requirement helps one obtain the characteristic function of $T_n$ via recursion on $e$ as previously stated. 

Now, given the following theorem:
\begin{thm} (Kleene [14]) \textit{For every recursive function $\phi_e(x)$ there exists some $e \in \omega$ such that
$$\phi_e(x)= \phi_{\psi(e)}(x)$$}
\end{thm}

We will choose to keep the function $\psi(e)$ to be an arbitrarily defined total computable function so as to formulate the following lemma.

\begin{lem} (Bauer [2]) \textit{There exists a partial function $\phi$ such that, if one is provided with a total unary function $\psi$, then there exists some $e \in \omega$ such that $\phi(e)\downarrow$ and $\phi(e) \neq \psi(e)$}.
\end{lem}
\begin{proof} Define $\phi(e)$ as follows
\[ \phi(e)=
  \begin{cases}
                                   0 & \text{if $\phi_{e}(e)\downarrow \ \text{and}\ \phi_{e}(e)\neq 0 $} \\
                                   1 & \text{if $\phi_{e}(e)\downarrow \ \text{and}\ \phi_{e}(e)\neq 1 $} \\
                                   \text{undefined} & \text{if $\phi_{e}(e)\uparrow$.}  \\
  \end{cases}
\]
Thus $\phi(e)$ is partial computable if $\phi_{e}(e)$ is defined. Assume $\phi(e)$ is computable. Since $\psi(e)$ is total, there exists an index $e \in \omega$ such that $\psi(e) = \phi_{e}(e)\downarrow$ holds. By definition, we see that the diagonal function $\phi_{e}(e)$ cannot be computable by values for all $e \in Dom(\phi)$ if we have $\phi(e)\downarrow = \phi_{e}(e)$. Therefore, we obtain the inequality $\phi(e)\downarrow \neq \phi_{e}(e)\downarrow = \psi(e)$, for if otherwise, then we may demonstrate that $\phi(e)$ is total (that is, it would define an enumeration of $TOT$ for all $e \in \omega$, such that $\psi(e)=\phi_{e}(e)$) which contradicts our assumption that $\phi(e)$ is computable.
\end{proof}

It will be important to note that the proof of Lemma 3.4 introduces the existence of a partial function $\phi(e)$ which, in a special sense, defines a diagonalization against the total computable functions for some input $e \in \omega$ on which a given total function $\psi(e)$ depends.

\begin{defn} \textit{Let $A,B \subseteq \omega$ be disjoint sets. We say that $A,B$ are computably seperable if there is some computable set $S$ such that $A \subseteq S$ and $B \cap S = \emptyset$. If $S$ is non-computable, we say that $A,B$ are computably inseperable and $S$ is a non-seperating set for the pair $(A,B)$.}
\end{defn}

\begin{prop} \textit{There exists a pair $(A,B) \subseteq \omega$ of computably enumerable, computably inseperable sets.}
\end{prop}

We are now in the position to formulate the concept of a diagonally non-computable majorizing function for an arbitrary class $\mathcal E$ of total computable functions: 

\begin{defn} \textit{Suppose $\psi$ is an arbitrary computable function belonging to a class $\mathcal E$ of functions defined on all well-founded initial segments of $Dom(R_e)$. Let $\left\{\phi_{e_i}\right\}_{i \in \omega}$ denote an increasing computable sequence on $\omega$. We write $\psi \ll_{\infty} F$ and say $``F$ diagonally majorizes $\psi''$ if there is some $k \in \omega$ for all $e \in Dom(\phi)$ such that}
\begin{enumerate}
\item For all $e \in Dom(\phi)$, $F(e) < \min{|t|}$ if $F(e) \neq \phi_{e_i}(e_i)\downarrow$, and $\min{|t|}$ is counted as the number $k \geq 2$ of prime factors of each $i \in \omega$ listed in the computable sequence $\left\{\phi_{e_i}\right\}_{i \in \omega}$ whenever $e +1 \leq k$ holds.
\item $F$ computes an index $e > 0$ which seperates $Dom(\phi)$ into a pair of finite, computably inseperable subsets of $\omega$ which are uniformly computable in $\phi_{e_i}(e_i)\downarrow$ if some $i \in \omega$ codes the characterstic function of $Dom(R_e)$.
\end{enumerate}
\end{defn}

Again, we recall that $\min{|t|}$ is used to denote the minimum length of the G{\"o}del number $t \in \omega$ which codes the finite sequence $\big \langle \sigma_0,\sigma_1,\ldots,\sigma_{n} \big \rangle$ of configuration states for our arbitrary computable function $\psi$ when simulated by some universal Turing machine $U$, where $t = 2^n \cdot \prod^{n-1}_{i=0} p^{\sigma_n}_{i+1}$, the number $p_{i}$ is the $i$th prime, and $0 \leq i < n$ [18, pp. 61]. In what follows, we will choose express any function belonging to our monotone enumeration of $\mathcal E$ as a uniformly computably enumerable set $S \subseteq \omega$ of (possibly unordered) pairs $(e,i) \in 2^{<\omega} \times \omega$. 

Furthermore, we intend to make the convention that, for any $i \in \omega$, the computable functions $\phi_{e_i} : \omega \longrightarrow Dom(R_e)$ are defined on all well-founded initial segments of $R_e(x,y)$ if the characteristic function for this arbitrary well-ordering is indeed computable. By adopting this convention, one is able to think of the sequence $\left\{\phi_{e_i}\right\}_{i \in \omega}$ as a class of total functions which is everywhere defined with respect to well-foundedness of the computable relation in question. \\

However, in order to explicate this notion of an arbitrary computable function $\psi$ which is everywhere defined with respect to a well-founded total computable predicate, we rely on the following Theorem:
\begin{thm} (Effective Transfinite Recursion)
(Riemann [19]) \textit{Let $R_e(x,y)$ be a well-founded, total computable predicate. Take $\psi(e) \in TOT$, and suppose that for every $i \in \omega$ and for every $x \in Dom(R_e)$:
$$\forall y < x, \ \phi_{e_i}(y)\downarrow \implies \phi_{\psi({e_i})}(x)\downarrow.$$
Then we have that $\exists e \in \omega$ such that 
$$\forall x \in Dom(R_e), \ \phi_{e}(x)\downarrow \ \text{and} \ \phi_{e}(x) \simeq \phi_{\psi(e)}(x).$$}
\end{thm}
\begin{proof} We follow the proof given in [19]. By appealing to the following proposition:
\begin{prop} (Sacks [24]) Suppose $R_{e}(x, y)$ is a well-founded, total computable predicate. If $\psi(e)$ is a total function and $e \in \omega$:
\begin{enumerate}
\item $\phi_{e}(y) \simeq  \phi_{\psi(e)}(x)$ for all $y < x \in Dom(R_e)$.
\item $\phi_{\psi(e)}$ is defined on every minimal element in the domain of $R_{e}(x, y)$.
\end{enumerate} 
Then we have that there is an index $e_i \in \omega$ such that $\phi_{e_i}=\phi_{\psi(e_i)}$ and $\phi_{e_i}$ is defined on all well-founded initial segments of $Dom(R_e)$.
\end{prop}

Suppose that if the function $\phi_{\psi(e)}$ were not defined on all $Dom(R_{e})$, then it would be so for some minimal
element $y \in Dom(R_{e})$. Apply Theorem 3.3 to obtain $\phi_{e_i} \simeq \phi_{e}$ . Then, by induction on $e_i$ for the
minimally defined function $\phi_{e_i}(y)$ shows that $\phi_{e}(x)\downarrow$, and thus $\phi_{\psi(e)}$ is defined on all the domain.
\end{proof}

Now, provided with a recursive equation of the form 
\begin{equation}
\forall n \in \omega, \ \psi_{e}(n) = \Phi(\psi \restriction n)
\end{equation}
we will refer to the index $e \in \omega$ as a $\textit{recursive extension}$ [23]. That is, we can acquire the index of a function $\phi_{\psi(e)}(x)$ which is defined on all $Dom(R_{e})$ as an extension from $\phi_{e}(x)$ through another index $e_i \in \omega$ of a function $\phi_{e_i}(y)$ which is defined on all well-founded initial segments of $Dom(R_{e})$. Additionally, we see that $\phi_{\psi(e)}(x)$ depends on $\phi_{e_i}(y)$ to be defined with respect to the $\textit{well-foundedness}$ of the domain in question. For example, let $\psi \restriction n$ denote $\psi$'s restriction to the set $\left\{ y : y < x \right\}$ if we let $x = n$ so that, given $x, y \in Dom(R_e)$, the relation $y <_e x$ defines an arbitrary well-ordering of $\omega$ when $\psi_e(n) \prec \psi_e(n)+1$ for all $e,n \in \omega$ and $\prec$ is linear.

Now, if some $e \in \omega$ serves as an index for the effective composition $\Phi(\psi \restriction n)$, then by Theorem 3.3 we know that there exists a fixed-point $e \in \omega$ for such an index. Consequently, we define the map $e \longmapsto \psi_{e}(n)\downarrow$ as the unique solution for each $\Phi,n$ in the recursive equation for an arbitrary well-ordering of $\omega$. With these ideas in mind, if we can futher suppose that $\Phi$ acts as an index for $\psi \restriction n$ and $\psi_{e}(n)$ is its fixed-point for each $n \in \omega$, then we say that an arbitrary computable function $\psi$ is defined by effective transfinite recursion if it constitutes a unique solution of (3.10).

\begin{lem} For any function $\psi$ belonging to a class $\mathcal E$ of computable functions defined on all well-founded initial segments of some total computable relation $R_e(x,y)$, where $\psi$ constitutes a unique solution to a recursive equation of the form $$\forall n \in \omega, \ \psi_{e}(n) = \Phi(\psi \restriction n)$$ there exists a function $F \in DNR_k$ such that $F$ diagonally majorizes $\psi$ in the sense of Definition 3.7, and $F$ is unique up to degree-isomorphism with any member belonging to $DNR_2$.
\end{lem}
\begin{rem} Before we demonstrate the lemma, we shall briefly expound on the ideas involved in its formulation. As we have seen from Proposition 3.9, if we are given any effectively enumerable class of functions, there exists a fixed-point $e_i \in \omega$ which can be applied to obtain a unique class of partial functions which are defined on all well-founded initial segments of the domain. Moreover, because it is not necessary that we restrict our attention to ``standard'' well-orderings of $\omega$ via the relation $R_e(x,y)$, we will make it a convention to allow $R_e(x,y)$ to define an arbitrary well-ordering of $\omega$ by setting $n=x$ for all $x > y$ belonging to the field of $R_e(x,y)$ if $\psi_e(n)$ is defined.
Within this perspective, we aim to define a  of the class of all functions defined by effective transfinite recursion with respect to the fact that, given any arbitrary class $\mathcal E$ of computable functions, it is possible to effectively enumerate $\mathcal E$ and obtain a computable function which majorizes every member belonging to $\mathcal E$ in the sense of Definition 3.1. In this way, we intuitively think of any unique, diagonally majorizing function as a ``global'', non-computable majorization of any computably bounded class of functions defined on all well-founded initial segments of $R_e(x,y)$. That is, our goal is to construct a computable function $\psi$ which is diagonally majorized by some unique $F \in DNR_k$ and defined on all $Dom(R_{e})$ in the sense of Theorem 3.8. To this end, we follow a construction of $F$ which is analogous to the construction of a partial function which is distinct from all total functions under an arbitrary indexing for a finite, non-empty set of arguments as accomplished in Lemma 3.4.
\end{rem}
\begin{proof} Let $\phi_e(x,y) = R_e(x,y)$. Suppose $\psi_{e}(n)$ acts as a fixed-point for $\Phi(\psi \restriction n)$ through the map $n \longmapsto \psi_e(n)$, and let the map be defined for all $n=x$ and for any $\Phi$, granted that the composition $\Phi(\psi \restriction n)$ is effectively specifiable. Let $\psi(e) = \phi_e(x,y)$. Throughout, we will have $Dom(\phi) \coloneqq \left\{ e \in \omega : \ \phi_e(e)\downarrow\right\}$. 

Define $DNR \coloneqq \left\{ F \in  2^\omega : \ \forall e \in \omega, \ F(e) \neq \phi_e(e)\downarrow  = \psi(e)\right\}$ be a non-empty class of total functions, and suppose we are given some $F \in DNR$ such that $F(e) < \min{|t|}$ whenever $\min{|t|}=$ the number $k \geq 2$ of prime factors of each $i \in \omega$. Thus, for arbitrarily large $i$, the function $F(e)$ is a $k$-bounded $DNR$ function when the number of factors of $i \in \omega$ are counted with multiplicity. 

From the lemma, we let $\mathcal E$ denote the class of one-one computable functions $\psi: \omega \longrightarrow \omega$ where $R_e(x,y)$ is a well-founded, total computable predicate. We say that $\mathcal E$ is computably bounded if there is a computable function $\phi_e=\phi_{\psi(e)}$ which majorizes every member of $\mathcal E$, and $\mathcal E$ satisfies the following conditions:

\begin{enumerate}[(i)]
\item There is a $y \in Dom(R_e)$ for every $x \in \omega$ such that, if $y <_e x$ induces an arbitrary well-ordering of $\omega$, then there is some $e \in \omega$ where $\phi_{e_i}(y)\downarrow \implies \phi_{\psi(e)}(x)\downarrow$.
\item For any $i \in \omega$ occuring in our enumeration, there is a computable sequence $\left\{\phi_{e_i}\right\}_{i \in \omega}$ which is uniformly computably enumerable in $\phi_{e_i}(e_i)\downarrow$.
\end{enumerate}

Assume that $\mathcal E$ is computably bounded so that, for every $x > y$ belonging to the field of $R_e(x,y)$, we have that $\phi_{e_i}(x) \ll \phi_{\psi(e)}(x)$ from Definition 3.1. If we identify any computable set with the characteristic function defined on its elements, then we may set $\phi_{\psi(e)}(x) \simeq \psi_{e}(n)$ if we have $n=x$ for any $x > y$, and so $\phi_{e_i}(n) \ll \psi_{e}(n)$.

Fix some $i \in \omega$ such that $0 \leq i \leq n$. From here on, we view the number $i$ as an index which codes the computable function $\psi$ that defines a computable linear-ordering $\prec$ of $\omega$ uniformly in some $e \in \omega$ such that $y <_e x = n$ implies that $\psi_e(n) \prec \psi_e(n)+1$ for every $n \in \omega$. Now, suppose we are given some $\hbar \in DNR_2$ which computes a number that seperates $Dom(\phi)$, and let $F(e) \in DNR_k$ be $F(e) = \hbar(\phi_e(x,y))$. We now distinguish two cases which we shall use to justify the claim that the finite initial segments of $Dom(\hbar)$ contain exactly the indices of a computable well-ordering of $\omega$ which is defined by $R_e(x,y)$ by fixing $n = x$ for all $x > y$ belonging to the field of $R_e(x,y)$.

\begin{case} $(k =2)$ If no $\psi \in \mathcal E$ defines an infinitely descending sequence $\left\{\phi_{e_i}\right\}_{i \in \omega}$ for some $i \in \omega$ such that $i = \phi_e(x,y)\downarrow$, then $F(e) \in DNR_{k+1}$ and there is some $\hbar \in DNR_2$ such that $F(e) \neq \phi_{e_i}(e_i)\downarrow$ and $F(e) \geq \hbar(i)$ for all $e \in \omega$.
\end{case}
\begin{case} $(k>2)$ If no $\psi \in \mathcal E$ defines an infinitely descending sequence $\left\{\phi_{e_i}\right\}_{i \in \omega}$ for any $i \in \omega$ such that $i = \phi_e(x,y)\downarrow$ and $i \in [0,n]$, then $F(e) \in DNR_{k}$ for all $k \geq 2$ and there is some $\hbar \in DNR_2$ such that $F(e) \neq \phi_{e_i}(e_i)\downarrow$ and $F(e) \geq \hbar(i)$ for all $e,i \in \omega$ with the exception of some fixed $i \in \omega$ such that $\phi(e_i) \neq \phi_{e_i}(e_i)\downarrow = \psi(e_i)$ by Lemma 3.4.
\end{case}

It is immediate from both cases that $DNR_k$ is upwards closed with respect to the degree of $\hbar(i)$ if no $\psi \in \mathcal E$ defines an infinitely descending sequence on $\omega$ and $\mathcal E$ is computably bounded, so if $Deg(F) \geq Deg(\hbar)$, then $Deg(F)$ is contained in $DNR_{k}$ for any $k \geq 2$.

Now it remains to show that, if there is some $F(e) \in DNR_k$ which seperates $Dom(\phi)$ into a finite, non-empty subset $S$ that contains the index for computing $R_e(x,y)$, then $Dom(\hbar)$ contains a number $i \in \omega$ which serves as the index of a computable well-ordering of $\omega$. Assume that $\psi \in \mathcal E$ defines an increasing computable sequence $\left\{\phi_{e_i}\right\}_{i \in \omega}$ on $\omega$ in stages, and each stage is numbered by some $n \in \omega$. Then, for any stage $\leq n$, we fix the requirement that, given a (possibly infinite) computably enumerable sequence $\left\{S_e\right\}_{e \in \omega}$, there are at least finitely many sets of computable inseperable pairs which are simultaneously listed in our increasing sequence, and $\bigcap_e S_e \in \left\{\phi_{e_i}\right\}_{i \in \omega}$ for all $e,i \in \omega$. Suppose we are given some $F^{\prime}(e) \in DNR_{k+1}$ for all $k \geq 2$ such that $F^{\prime}(e) = \hbar(\phi(e_i))$. Now, for any $k > 2$, because we have excluded some fixed index $i \in \omega$ such that $\phi(e_i) \neq \psi(e_i)$, then we cannot effectively seperate any $i \in S_e$ which is distinct from any number in $\left\{\phi_{e_i}\right\}_{i \in \omega}$ that is uniformly computable in $\phi_{e_i}(e_i)\downarrow$. Set $\hbar(i) = \hbar(\phi_e(x,y))$ and fix $\phi_e(x,y)\downarrow = \psi_e(n)$ when $0 \leq i \leq n$. Now, let $\hbar(i)$ denote the characteristic function which seperates $Dom(\phi)$ into a finite subset $S \coloneqq \left\{ (e,i): \forall i \in [0,n], \phi(e_i) \neq \phi_{e_i}(e_i)\downarrow \right\}$ which is disjoint from $\left\{ i \in \omega: \psi(e_i)\downarrow \right\}$ if $\hbar(\phi(e_i)) \neq \phi_{\phi(e_i)}(\phi(e_i))\downarrow$ and $Deg(\hbar(\phi(e_i))) \equiv Deg(\hbar(i))$. On the hypothesis that the class $DNR_{k+1}$ is upwards closed for any $k \geq 2$, then $F^{\prime}(e)$ acts as the characteristic function of $Dom(\phi)$ if $Deg(F^{\prime}(e)) \geq Deg(\hbar(i))$. Now fix $F(e) = S_e$ for all $e \in Dom(\phi)$ such that $S_e \supseteq S_{e+1}$ for $e +1 \leq k$. By induction on $e$ when $e+1$ is counted as the number of prime factors for any stage $\leq n$, we see that $F(e) \in DNR_{k+1}$ seperates $Dom(\phi)$ with respect to the degree of $\hbar(i)$ if a non-seperating set $S = \bigcap_e S_e$ of indices $i \in \omega$ which code a computable well-ordering of $\omega$ coincides with the finite initial segments of $Dom(\hbar)$ infinitely often.

Inductively, we have verified that $Deg(F^{\prime}(e)) \equiv Deg(F(e))$ if $F^{\prime}(e)$ is bounded by a constant number $k$ for any $i \in \omega$ listed in our computable sequence $\left\{\phi_{e_i}\right\}_{i \in \omega}$, so $F(e) \in DNR_{k+1}$ is unique up to degree for $k = 2$. On the hypothesis that $DNR_k$ is upwards closed for all $k \geq 2$, we observe that the $e$th computably enumerable set $S_e$ in our sequence is seperated by $F(e)$ infinitely often up to the degree of $\hbar(i)$ for any stage $\leq n$. Thus, for any $i \in [0,n]$, we see $i \in Dom(\hbar)$ implies that $i \in S_e$ for all $e \in Dom(\phi)$ such that $F(e) \neq \phi_{e_i}(e_i)\downarrow$, and our claim follows as desired.

Finally, in order to demonstrate the uniqueness of $F(e) \in DNR_{k}$ for all $k \geq 2$ with respect to some $\hbar(i) \in DNR_2$, we rely on the following theorem:
\begin{thm} (Jockusch and Soare [10 pg. 195])
{For each $k \geq 2$, the degrees of members belonging to $DNR_k$ coincide with the degrees of members belonging to $DNR_2$ up to degree-isomorphism}.
\end{thm}
We can now claim that $Deg(F) \equiv Deg(\hbar)$ for any $k \geq 2$. Thus, we see that $\psi(e) \ll_{\infty} F(e)$ holds when $Deg(F(e)) \equiv Deg(\hbar(\psi(e)))$ for all $e \in Dom(\phi)$. Therefore, $F(e) \in DNR_k$ is unique up to degree-isomorphism with $\hbar(i) \in DNR_2$ by induction on the pair of $e$ and the length of $i$ when $0 \leq i \leq n$ and $n = x$. This concludes the proof of Lemma 3.11.
\end{proof}

\section{$\Pi^0_1$-Classes of Diagonally Non-Computable Majorizing Functions}

We now turn to concept of a $\Pi^0_1$-class. Recall that such classes are closed under initial substring
when provided with an appropriate G{\"o}del numbering [5]. Formally,
we say that a set $\mathcal{H} \subseteq 2^{\omega}$ is a $\Pi^0_1$-class if we are able to put $\mathcal{H}$ in the following form $$ \mathcal{H} \coloneqq \left\{ F \in 2^{\omega} \ \colon \ \forall x, \ V(F\restriction x)\right\}$$
where $V$ is a computable relation. 

Following [9], we are able to extend our intuition about these classes in a way that allows us to think of a $\Pi^0_1$-class as the set of infinite paths through a computable subtree of $2^{<\omega}$. To make this applicable to our purposes, we form a set of strings (under a suitable G{\"o}del numbering) which include the diagonally majorizing functions that were introduced in Definition 3.7.
By Lemma 3.11 there is a non-computable function $F(e) \geq \hbar(i)$ which is constructed so that the string $\hbar(i)$ satisfies the properties of a diagonally majorizing function for a class $\mathcal E$ of computable functions defined on all well-founded initial segments of some arbitary computable well-ordering of $\omega$. Owing to the uniqueness of $F(e) \in DNR_k$ up degree-isomorphism with some $F^{\prime}(e) \in DNR_{k+1}$, if $Deg(F^{\prime}(e)) \equiv Deg(F(e))$ and $DNR_k$ is upwards closed for all $k \geq 2$, then $\hbar(i)$ is of computably enumerable degree.  

Most importantly, however, is that fact that $DNR_2$ can be viewed as a non-empty, computably bounded $\Pi^0_1$-class of functions which possess a computably enumerable degree. Additionally, given any $\Pi^0_1$-class $DNR_2$, one can construct out of $DNR_2$ a set $\mathcal C$ which is closed under initial substring and computably enumerable in a member of low degree\footnote{For example, since we have that $DNR_2$ is computably bounded, then the computable function which majorizes each member belonging to $DNR_2$ will be computable in a member of low degree.}. Finally, if the binary relation $V$ defined on $\mathcal C$ is computable, then $\mathcal C$ can be thought of as a computable tree containing the class of all infinite paths through $DNR_2$. \\

With these ideas in mind, we aim to construct a canonical system $\mathcal{O}^{\mathbb N}$ for notations of recursive
ordinals with respect to the uniqueness of the diagonally majorizing functions. In particular, we will pay special attention to the fact that these majorizing functions are applied to a class of computable functions which are defined on all well-founded initial segments of some arbitrary well-ordering of $\omega$ and compute indices of computable well-orderings of $R_e(x,y)$ up to degree-isomorphism. \\

That is, by the fact that, for any $\psi$ which constitutes a unique solution of (3.10) and belongs to the class $\mathcal E$, there is a function $F(e) \in DNR_k$ such that $F(e)$ diagonally majorizes $\psi(e)$, then we can demonstrate that $F(e)$ is unique up to degree-isomorphism with some string $\hbar(i) \in DNR_2$. Now, with respect to the string $\hbar(i)$, we construct $\mathcal{O}^{\mathbb N}$ as the set of infinite paths through $DNR_2$ which uniquely coincide with our diagonally majorizing function $F \in DNR_k$ up to degree-isomorphism.
Accordingly, we provide a formal definition of $\mathcal{O}^{\mathbb N}$ as follows. Let $\hbar \restriction e$ denote $\hbar(i)$'s restriction to the set of numbers which compute a recursive well-ordering $R_e(x,y)$ of $\omega$ with $<_e \ \coloneqq R_e$. Then we define $\mathcal{O}^{\mathbb N}$ as follows: $$\mathcal{O}^{\mathbb N} \coloneqq 
\left\{\hbar  \in 2^{\omega} \ \colon \ \forall e \forall y <_e x, \ L(\hbar\restriction e, y)\right\}$$
where $L \subseteq 2^{<\omega} \times \omega$ is a computable binary relation and $y <_e x$ defines an arbitrary recursive well-ordering of $\omega$ if $x =n$ whenever $\psi_{e}(n)$ is defined from equation (3.10) as the characteristic function for the well-ordering in question.

Furthermore, by reference to the fact that the finite initial segments of the domain of the string $\hbar \in DNR_2$ are precisely the numbers contained in the set $\mathcal{W}$ of indices of recursive ordinals, then we find that $\mathcal{O}^{\mathbb N}$ is a $\Pi^{0}_1(\mathcal{W})$-class, in the sense that the relation $ L(\hbar\restriction e, n)$ is computable in $\mathcal{W}$ by the 
reducibility relation $\leq_T$, and we have that $\mathcal{W}$ is $\Pi^{1}_1$-complete [3, pg. 4]. 

Naturally, it is desirable to define the notion of what it means for a given recursive ordinal to possess a notation in $\mathcal{O}^{\mathbb N}$ which reflects the intuitive idea of an ordinal notation belonging to Kleene's $\mathcal{O}$. Let $|\alpha| \in \mathcal{O}^{\mathbb N}$ denote the notation for a recursive ordinal $\alpha < \omega^{CK}_1$. To build on our intuition regarding the ordinal notations in Kleene's $\mathcal{O}$, we will make the convention to view any recursive ordinal $\alpha$ as a certain well-founded tree which is order-isomorphic to some recursive well-ordering defined by $R_e(x,y)$ of length $\alpha$. Hence, in pursuit of the idea that we form $\mathcal{O}^{\mathbb N}$ as the class of infinite paths through $DNR_2$ with respect to the uniqueness of our diagonally majorizing function $F(e) \in DNR_k$, we define the concept of a path $\mathcal P$ through $\mathcal{O}^{\mathbb N}$ as a subset of $\mathcal{O}^{\mathbb N}$ such that $\mathcal P$ is closed under initial substring and computable in a member $\hbar \in DNR_2$ of computably enumerable degree. Formally,
$$\mathcal P \coloneqq \left\{\hbar  \in 2^{\omega} \ \colon \ \forall e \forall y<_ex ,\ L(\hbar\restriction e, y)\right\}.$$
Consequently, if we let $y <_e x$ with $x =n$ be an arbitrary recursive well-ordering defined by $R_e(x,y)$ as before, then the path $\mathcal P \subseteq \mathcal{O}^{\mathbb N}$ induces a well-ordering defined by $L \subseteq 2^{<\omega} \times \omega$ if $L$ is uniformly computably enumerable in some finite subset of $\mathcal W$. Moreover, for each finite initial segment of $Dom(\hbar)$, the well-ordering defined by $R_e(x,y)$ is coded by some fixed index $i \in \hbar\restriction e$. In particular, since the binary relation $L(\hbar\restriction e, y)$ is computable in $\mathcal W$, then we say that a path $\mathcal P$ represents a unique notation $|\alpha| \in \mathcal{O}^{\mathbb N}$ for some recursive ordinal $\alpha < \omega^{CK}_1$ if the linear-ordering of all intitial substrings in $\mathcal P$ possesses a standard well-ordered copy $|\alpha|$ which is definable over $L$ and uniformly computably enumerable in some finite subset of $\mathcal W$, and $|\alpha|$ is order-isomorphic with $\alpha$ via a one-one map\footnote{Note that this order-preserving map will be defined by effective transfinite recursion on $<_e$.} from the field of $<_e$ into $L$.

\begin{defn}
Following the outline in [21, pg. 48], we now define by transfinite recursion on the order-type $\alpha$ of an arbitrary recursive well-ordering $R_e(x,y)$ of $\omega$
\begin{enumerate}
\item A path $\mathcal P_{\alpha} \subseteq \mathcal{O}^{\mathbb N}$ which is linearly ordered under initial substring.
\item A binary computable relation $<_{\mathcal{O}^{\mathbb N}} \ \subseteq 2^{<\omega} \times \omega$.
\end{enumerate}
Given two distinct recursive well-orderings $\alpha, \beta$ obeying the relation $\alpha <_e \beta$, we let the paths $\mathcal P_{\alpha}$ and $\mathcal P_{\beta}$ denote distinct, linearly ordered subsets of $\mathcal{O}^{\mathbb N}$ such that $|\alpha| <_{\mathcal{O}^{\mathbb N}} |\beta|$ implies that $|\alpha| \in \mathcal P_{\alpha}$ and $|\beta| \in \mathcal P_{\beta}$ for any $\alpha <_e \beta$. From the distinctness of $|\alpha|, |\beta|$ it follows that any $|\alpha| \in \mathcal{O}^{\mathbb N}$ uniquely corresponds with the recursive ordinal $\alpha < \omega^{CK}_1$ under the ordering of $<_e$. Let $\hbar \restriction e$ be as above, where $\hbar$'s restriction to $e$ satisfies the binary relation $L(\hbar \restriction e, y)$ for all $y \in Dom(<_e)$ such that $y <_e x$ implies that $\psi_e(x) \prec \psi_e(x)+1$ is a computable well-ordering of $<_e$ coded by some $i \in \omega$. Define, respectively, the path representations for $\emptyset$ and $\alpha + 1$:
$$ \mathcal P_{\emptyset} \coloneqq \left\{\emptyset\right\}, \ \mathcal P_{\alpha+1} \coloneqq \left\{ \hbar \in 2^{\omega} : \hbar\restriction 2^e \in \mathcal P_{\alpha} \iff i \in \hbar\restriction 2^e \wedge (\forall i <_e 2^e, \ L(\hbar\restriction 2^e, i)) \right\}$$
$$ <_{\mathcal P_{\emptyset}} \ \coloneqq \emptyset, \ <_{\mathcal P_{\alpha+1}} \ \coloneqq \left\{ (i,\hbar\restriction 2^e) : \ \hbar\restriction 2^e \in \mathcal P_{\alpha} \iff i \in \hbar\restriction 2^e \wedge (i <_e 2^e \vee i = 2^e) \right\}$$
Now, for the case that we have a path $\mathcal P_{\gamma}$ which represents a unique notation for a recursive limit ordinal $\gamma < \omega^{CK}_1$,
$$\mathcal P_{\gamma} \coloneqq \left\{ \hbar \in 2^{\omega} : \hbar\restriction 3^e \in \mathcal P_{\alpha} \iff i \in \hbar\restriction 3^e \wedge (\forall i <_e 3^e, \ L(\hbar\restriction 3^e, i)) \right\}$$
$$<_{\mathcal P_{\gamma}} \ \coloneqq \left\{ (i, \hbar\restriction 3^e) : \ \hbar\restriction 3^e \in \mathcal P_{\gamma} \iff i \in \hbar\restriction 3^e \wedge (\forall e \in \mathcal W, \ i <_e 3^e )\right\}.$$

Now, we have $$\mathcal{O}^{\mathbb N} \coloneqq \bigcup_{\alpha \in \omega^{CK}_1} \mathcal P_{\alpha} \ \text{and} \ <_{\mathcal{O}^{\mathbb N}} \ \coloneqq \bigcup_{\alpha \in \omega^{CK}_1} <_{\mathcal P_{\alpha}}.$$
\end{defn}

\begin{thm}
Every recursive ordinal $\alpha < \omega^{CK}_1$ has a notation $|\alpha| \in \mathcal{O}^{\mathbb N}.$
\end{thm}
\begin{proof} Suppose $\gamma$ is the least recursive ordinal that does not have a unique notation $|\gamma| \in \mathcal{O}^{\mathbb N}$. By definition, a notation for $\gamma$ is a uniformly computably enumerable well-ordering which is represented by the path $\mathcal P \subseteq \mathcal{O}^{\mathbb N}$ of length $\gamma$. Throughout, we use $\gamma$ as an arbitrary limit ordinal defined as $\lim_{n \rightarrow \infty} \alpha_n$.
Suppose we are given some path $\mathcal P_{\gamma} \coloneqq \left\{\hbar  \in 2^{\omega} \ \colon \ \forall e \forall i<_e 3^e, \ L(\hbar\restriction 3^e, i)\right\}$ which represents $|\gamma| \in \mathcal{O}^{\mathbb N}$ as a computable well-ordering of $2^{<\omega} \times \omega$. Assume $R_e(x,y)$ defines a well-ordering with order-type $\gamma$, the number $3^e \in \mathcal W$ denotes the index for $\gamma$, and $\phi_e(x,y)\downarrow = R_e(x,y)$. Essentially, our argument adapts the proof of Corollary 5.5 in [24, pg. 20].
\begin{lem} (Schwichtenberg [23]) Let $<_e$ be a well-founded partial ordering of some set $Dom(R_e) \subseteq \mathbb N$, and $Q = 2^{<\omega} \times \omega$ a binary relation. Assume there is a partial function $\phi$ such that, for every $n = x \in Dom(R_e)$ and $e \in \omega$:
$$\forall y <_e n, \ Q(y, \varphi_{e}(y)) \implies Q(n, \phi(e,n)).$$
Then there exists a partial function $\varphi$ such that $\forall n \in Dom(R_e)$,\ $Q(n, \varphi(n))$. 
Moreover, if $\phi \in TOT$, then $\varphi \in TOT$.
\end{lem}
On the hypothesis that $\varphi \in TOT$ for each $y <_e n$ belonging to the field of $<_e$, fix $n = x$ for all $x \in Dom(<_e)$. Now let $\phi_e(x,y)\downarrow = \phi(e,n)$. If $y <_e n$ implies that $\psi_e(n) \prec \psi_e(n)+1$, then by well-founded induction on $<_e$ for all $y <_e n$, $$Q(y, \varphi_{e}(y)) \implies Q(n, \phi_e(x,y)).$$
Let $\mathcal{W}$ denote the set of indices of recursive ordinals and let WF$(R_e)$ denote $\left\{ e \in \omega : R_e(x,y) \ \text{is well-founded} \right\}$ when $R_e(x,y)$ defines a recursive well-ordering with order-type $\gamma$. Without loss in generality, we assume that $\mathcal{W} \leq_m$ WF$(R_e)$ since WF$(R_e)$ is $\Pi^{1}_1$-complete. To derive our contradiction, we now rely on the notion of the height of a computably enumerable, well-founded relation [24, pg. 16]. Assume the relation $L$ defines a well-ordering of $\omega$ and the height $[R_e]$ of $R_e(x,y)$ is equal to $\gamma$. Suppose further that $[L] \leq [R_e]$ so that we may consider $L$ as a uniformly computably enumerable ordering based on our well-foundedness assumptions. Let $[L] \neq [Q]$ and have that $[R_e] \leq [3^e]$. By the fact that every well-ordering $\leq [3^e]$ is recursive, then we may claim that $[L]$ is computable in a uniform manner, since the ordering induced by $L$ is computable in $\mathcal W$ by definition. Now, let $i \in \hbar \restriction 3^e$ if, for all $e \in \omega$ such that $[L] \leq [R_e]$, we can decide that $\hbar\restriction 3^e \in \mathcal P_{\gamma}$ on the hypothesis that $L$ is uniformly computable in each subset $X \subseteq \mathcal W$.

\begin{cor} (Sacks [24, pg. 20]) WF$(R_e) \notin \Sigma^{1}_1$.
\end{cor}
Let $\gamma(x)$ be an ordinal variable and let $\varphi \in TOT$ be a computable function which many-one reduces $X \subseteq \mathcal{W}$ to WF$(R_e)$. If $\gamma$ is infinite and $3^e \in X$, then we have the relation $\forall \gamma\exists y, R(3^e,y,\gamma(x))$ holds with $R$ computable and $X$ the projection of $R$. Suppose $\varphi(e) = \varphi_e(y)\downarrow$ if and only if $\forall \gamma\exists y, R(3^e,i,y,\gamma(x))$, and let $\varphi_e(y) \in TOT$ for all $y <_e x = n$. Now, we can deduce from the assumption that $\gamma$ lacks a unique notation for all $e \in \mathcal W$ the fact that $i \in \hbar\restriction 3^e$ if $\varphi_e(y)\downarrow$ and $\forall \gamma\forall y, R(3^e,i,y,\gamma(n))$ holds for some $e \in \omega$ such that $[R_e] \leq [3^e]$. That this implies our claim follows from the condition that $L$ is computable in $\mathcal W$, so $\hbar\restriction 3^e \in \mathcal P_{\gamma}$ if $3^e \in X$ and hence $[L] \leq [3^e]$ by transitivity of $\leq$. Therefore, $\varphi_e(y)\downarrow$ if $\varphi$ reduces $\forall \gamma\exists y, R(3^e,y,\gamma(n))$ to WF$(R_e)$ for some $e \in \mathcal W$ such that $[L] \leq [3^e]$. This demonstrates that $\varphi(e) \in WF(R_e)$ if and only if WF$(R_e) \in \Sigma^{1}_1$, on the assumption that $\gamma$ is infinite and is not uniquely associated to some notation $|\gamma| \in \mathcal{O}^{\mathbb N}$ represented by $\mathcal P_{\gamma}$.
\end{proof}

\section{Hierarchical Classification of the Computable Functions on Paths Through $\mathcal{O}^{\mathbb N}$}

As shown, to make the claim that one is able to construct a natural system of notations for
recursive ordinals is merely to exploit the intuitive belief that, relative to their complexity, certain decision procedures are of a
higher difficulty than others. That is, our belief in the existence of any “natural” hierarchy of computable
functions should be grounded in the ability to precisely distinguish, at any fixed level of the hierarchy, the
ordinal complexity of any two distinct, arbitrarily given decision procedures with respect to some
verifiably computable binary relation $L \in \Pi^{0}_1$. By using canonical notations in $\mathcal{O}^{\mathbb N}$ for the recursive ordinals, we are able to preserve the intuitiveness of the idea that any hierarchy of computable functions is linear and everywhere defined if our indexing of the hierarchy define ``natural'' well-orderings of $\omega$. Owing to this outline, one is able to claim that the notations in $\mathcal{O}^{\mathbb N}$ are precisely ``natural'' in the sense that they provide an optimal measure of the ordinal complexity of the computable functions with respect to verifying the totality of $R_e(x, y)$ as an arbitrary well-ordering of $\omega$.

In particular, we are now in the position to analyze the properties that a path $\mathcal P$ through $\mathcal{O}^{\mathbb N}$ should possess. We follow the outline [22, pg. 7]:

\begin{enumerate}
\item $\mathcal P$ is linearly ordered and closed under predecessors with respect to $<_{\mathcal{O}^{\mathbb N}}$.
\item If $|\beta| \in \mathcal P$ and $|\alpha| <_{\mathcal{O}^{\mathbb N}} |\beta|$, then $|\alpha| \in \mathcal P$.
\item Every recursive limit ordinal $\gamma < \omega^{CK}_1$ uniquely corresponds to some notation $|\gamma| \in \mathcal P$ and no other.
\end{enumerate}

We may now attempt analyze the structural properties that a canonical classification of the computable functions should possess if indexed by some $\mathcal P \subseteq \mathcal{O}^{\mathbb N}$. Let $\mathcal E^{\infty}$ denote the class of general recursive functions and let $(\phi_{|\alpha|})_{|\alpha| \in \mathcal P}$ denote our intended hierarchy of computable functions where $\phi \in \mathcal E^{\infty}$, then
\begin{enumerate}[(I)]
\item For every $|\alpha|,|\beta| \in \mathcal{O}^{\mathbb N}$, we have that $|\alpha| \neq |\beta| \implies \phi_{|\alpha|} \neq \phi_{|\beta|}.$
\item For any $|\alpha|\in \mathcal P$, $\phi_{|\alpha|} \in TOT$.
\item For all $\phi \in TOT$, there is some unique $|\alpha| \in \mathcal P$ such that $\phi_{|\alpha|} = (\phi_{|\alpha|})_{|\alpha| \in \mathcal P}$.
\end{enumerate}

With these details in place, we expect that the statement $Q(|\alpha|,x,y)$ which expresses the property that the $|\alpha|^{th}$ function with input $x$ taking the value $y$ should be inductively definable over some computable binary relation $L \in \Pi^{0}_1$ $$\forall e \forall y<_ex ,\ L(\hbar\restriction e, y)$$ with $|\alpha|$ a unique notation. Naturally, this follows from Defintion 4.1 of $\mathcal{O}^{\mathbb N}$. Moreover, this $\Pi^{0}_1$-definable statement $Q$ should ``intensionally'' reflect the $\Pi^{1}_1$-statement $$\forall \alpha\exists y, R(e,x,y,\alpha)$$ when $e \in \mathcal W$, and $\alpha$ denotes an ordinal variable which is uniquely associated to the notation $|\alpha| \in \mathcal P$ by condition (3). 
Assume that there is a computably enumerable relation $<^{\prime}_{\mathcal{O}^{\mathbb N}}$ which extends $<_{\mathcal{O}^{\mathbb N}}$ such that, for any $|\beta| \in \mathcal{O}^{\mathbb N}$, we have $$|\alpha| <_{\mathcal{O}^{\mathbb N}} |\beta| \iff |\alpha| <^{\prime}_{\mathcal{O}^{\mathbb N}} |\beta|.$$ Now, we inductively define $<^{\prime}_{\mathcal{O}^{\mathbb N}}$ as follows:
\begin{enumerate}[(a)]
\item If there exists some $|\alpha|$ such that $|\alpha| \neq 0$, then $|\alpha| <^{\prime}_{\mathcal{O}^{\mathbb N}} |\alpha|+1$.
\item If there exists some $|\alpha|,x,y$ such that $Q(|\alpha|,x,y)$ and $\phi_{|\alpha|}(x)=y$, then we have that $|\alpha| <^{\prime}_{\mathcal{O}^{\mathbb N}} 3^{|\alpha|}$ where $3^{|\alpha|}$ is a unique limit notation.
\item $<^{\prime}_{\mathcal{O}^{\mathbb N}}$ is transitive.
\end{enumerate}
Thus we see that $Q$ is $\Sigma^{0}_1$-definable and monotone, so $<^{\prime}_{\mathcal{O}^{\mathbb N}}$ is inductively definable over $Q$ for some $|\alpha| \in \mathcal P$ and $\phi \in TOT$. Owing to the fact that $L$ is computable in $\mathcal O$, then \textit{a fortiori}, it should not be the case that there exists a $\Sigma^{1}_1$-definition of $\mathcal P$ with respect to $<^{\prime}_{\mathcal{O}^{\mathbb N}}$ being inductively definable over $Q$, which is actually the case if one assumes that $\alpha$ is infinite by Theorem 4.2:

\begin{cor} There is no $\Sigma^{1}_1$-definable computable predicate $Q$ which extends $<_{\mathcal{O}^{\mathbb N}} \ \subseteq 2^{<\omega} \times \omega$ in the sense that, if $\alpha,\beta$ are infinite ordinals, then $$|\alpha| \in \mathcal P \iff \exists \phi \in TOT \ \forall |\beta| \in \mathcal{O}^{\mathbb N} $$ such that $$\forall x > y \in Dom(R_e), \ Q(|\beta|,x, \phi(x)) \implies |\alpha| <^{\prime}_{\mathcal{O}^{\mathbb N}} |\beta|.$$
\end{cor}

Following [20, pg. 5], suppose that $\mathcal E(L)$ denotes the class of unnested $L$-computable functions with respect to the fact that $L$ is computable in $\mathcal O$ where $\mathcal O \in \Pi^{1}_{1}$. Let $[L]$ represent the height of $L$ as in Theorem 4.2, which by definition is identified as a well-ordering of $\omega$. Thus, if we claim that this well-ordering induced by $[L]$ is uniquely associated with the computable well-ordering defined by some total, well-founded relation $R_e$, then we have that $[L] \leq [R_e]$ uniformly in $e \in \omega$. 

Let $L$ be as above, and suppose that $L^{\prime}$ represents the extension of $L$ with respect to $<_{\mathcal{O}^{\mathbb N}}$. Now, if $L, L^{\prime}$ range over ``natural'' orderings of $\omega$ induced by $<_{\mathcal{O}^{\mathbb N}}$, then
\begin{enumerate}[(A)]
\item If $[L] \leq [L^{\prime}]$, then $\mathcal E(L) \subseteq \mathcal E(L^{\prime})$.
\item If $[L^{\prime}] > [L]$, then $\mathcal E(L) \neq \mathcal E(L^{\prime})$.
\item $\mathcal E^{\infty} = \bigcup \mathcal E(L)$.
\end{enumerate}

Consequently, we observe that our intended classification of $\mathcal E^{\infty}$ on paths through $\mathcal{O}^{\mathbb N}$ is indeed a ``natural'' classification of $\mathcal E^{\infty}$ by generating recursions over the orderings induced by $<_{\mathcal{O}^{\mathbb N}}$ to succesively define larger classes of unnested $L$-computable functions of increasing complexity.

However, because one has that the relation $<_{\mathcal{O}^{\mathbb N}}$ is arithmetical, one may suspect that there exists a certain ``absolute'' classification of $\mathcal E^{\infty}$ on paths through $\mathcal{O}^{\mathbb N}$ with respect to the property that $\mathcal E^{\infty}$ is closed under relative computability. To see this, we may suppose that $Dom(<_{\mathcal{O}^{\mathbb N}})$ belongs to the class of sets which are inductively definable from a computable binary relation $L$ and comptuable in $\Pi^{0}_{n}$ for any $n \geq 1$. By reflecting on the definition of $<_{\mathcal{O}^{\mathbb N}}$ and taking the convention that any successive $n$-fold extension of the relation $<_{\mathcal{O}^{\mathbb N}}$ is $\Delta^{0}_{n+1}$-complete with respect to the binary relation $L$ being Turing-reducible in Kleene's $\mathcal O$, one can sketch how extensions of $<_{\mathcal{O}^{\mathbb N}}$ corresponds with arithmetical definability at some stage. In particular, we may aim to associate any $\Delta^{0}_{n+1}$-definable extension of $<_{\mathcal{O}^{\mathbb N}}$ in a manner that reflects a successive $n$-fold jump iteration $\emptyset^{(n)}$ of the computable sets $\emptyset^{(0)}$, which we define inductively as follows: $$\emptyset^{(0)} \coloneqq \emptyset, \ \emptyset^{(n+1)} \coloneqq (\emptyset^{(n)})^{\prime}$$ Naturally, any $n^{th}$-iteration of the jump is $\Sigma^{0}_n$-definable by a direct induction on $n \in \omega$. Therefore, because we have made the convention to have $<_{\mathcal{O}^{\mathbb N}} \ \in \Delta^{0}_{n}$, then any extension $<^{\prime}_{\mathcal{O}^{\mathbb N}}$ is $\Delta^{0}_{n+1}$-complete if and only if it is computable in $\Sigma^{0}_n$.

Therefore, it appears worthwhile to investigate the various arithmetically definable $\textit{closure}$ properties that such an ``absolute'' classification of $\mathcal E^{\infty}$ would possess with respect to iterating the jump operator through the effective transfinite for each unique notation. Moreover, because it follows from clauses (A-C) that there exists a ``natural'' classification of the computable functions indexed by unique notations in $\mathcal{O}^{\mathbb N}$, then one may suspect that there exists a level of collapse which is obtained by defining a transfinite iteration of the jump operator along paths through $\mathcal{O}^{\mathbb N}$.

In a sense, the existence of a collapsing level in the classification of the computable functions on paths through $\mathcal{O}^{\mathbb N}$ would mirror the fact that in Kleene's $\mathcal O$, there exists a $\Pi^{1}_1$-definable subset $X$ of $\mathcal O$ which is linearly ordered by $<_{\mathcal{O}}$ and of order-type $\omega^{CK}_1$ [23, pg. 10]. 

However, as is known, the system $\mathcal O$ is infinitely branching at recursive limits $\geq \omega$, and so such a classification with respect to $X$ fails in any absolute manner. Now if we wish to iterate the jump operator into the effective transfinite for any $\alpha < \omega^{CK}_1$, our hope is to rely on the fact that every recursive ordinal possesses a unique notation in $\mathcal O^{\mathbb N}$ with respect to the relation $<_{\mathcal{O}^{\mathbb N}}$ being inductively definable over a computable binary relation $L$, such that $L$ is reducible in $\emptyset^{(\alpha)}$.

We now sketch our means of defining this transfinite iteration of the jump operator along paths through $\mathcal{O}^{\mathbb N}$. Suppose $\lambda < \omega^{CK}_1$ is a recursive limit ordinal which possess a canonical notation $|\lambda| \in \mathcal{O}^{\mathbb N}$, and this notation is represented by a binary relation $L$ on a path $\mathcal P \subseteq \mathcal{O}^{\mathbb N}$ which is computable in some $X \in 2^{\omega}$ such that $<_{\mathcal{O}^{\mathbb N}}$ is inductively definable over $L$. Then we define $\emptyset^{(\lambda)}$ as the $\lambda$th jump of $X$ by iterating the Turing jump at successor stages $\alpha +1$ and taking an effective limit at stage $\lambda$ if there is some $\phi \in TOT$ such that $\phi = \phi_{|\lambda|}$ for some unique $|\lambda| \in \mathcal P$. Note that this would follow from (III) if we have that $\phi$ is an unnested $L$-computable function. Furthermore, since we may appeal to the fact that the jump operator is degree invariant with respect to $L$ being reducible in $\emptyset^{(\lambda)}$, then it is clear that the naturalness of iterating the jump operator along paths through $\mathcal{O}^{\mathbb N}$ is also preserved at limit stages in the order-theoretic sense, and therefore one can classify $\mathcal E^{\infty}$ on paths through $\mathcal{O}^{\mathbb N}$ independently of the ``naturalness'' of the recursions defined over the orderings induced $<_{\mathcal{O}^{\mathbb N}}$ of length $\geq \omega$. \\

Now, in order to provide a more formal setting for discussing the meaning of a collapsing result for an absolute classification of $\mathcal E^{\infty}$, we provide the following definition:
\begin{defn} Let $\mathcal P \subseteq \mathcal{O}^{\mathbb N}$ be a path through $\mathcal{O}^{\mathbb N}$ which is closed under initial substring. We say $\mathcal P$ is maximal with respect to $<_{\mathcal{O}^{\mathbb N}}$ if every initial segment of $\mathcal{O}^{\mathbb N}$ is contained in $\mathcal P$ under the ordering of $<_{\mathcal{O}^{\mathbb N}}$ with respect to predecessors.
\end{defn}

By proceeding inductively through the class of constructive ordinals, one might conjecture that this collapsing result would occur at the degree $\emptyset^{(\omega^{CK}_1)}$ of Kleene's $\mathcal O$, in the sense that there exists a maximal path $\mathcal P \subseteq \mathcal{O}^{\mathbb N}$ such that, for any $\phi \in TOT$, we have that $\phi = (\phi_{|\lambda|})_{|\lambda| \in \mathcal P}$ for some limit $\lambda$, and the binary computable relation $L$ on $\mathcal P$ is reducible to $\emptyset^{(\omega^{CK}_1)}$ when $\mathcal{O}^{\mathbb N}$ acts as a $\Pi^{0}_{1}(\mathcal O)$-class.
\\

\subsection*{Concluding Remarks}

I would like to thank Professor Solomon Feferman for reviewing an earlier draft of this manuscript and providing encouragement in the process of its completion. Additionally, I would like to thank Professor Jeffry Hirst for looking over the typesetting and exposition.

\end{document}